\documentclass[12pt]{amsart}

\usepackage{amssymb}

\newtheorem{theorem}{Theorem}

\newtheorem{lemma}[theorem]{Lemma}
\newtheorem{proposition}[theorem]{Proposition}

\theoremstyle{definition}
\newtheorem{definition}[theorem]{Definition}

\theoremstyle{remark}

\numberwithin{equation}{section}

\newcommand{\dbar}  {\bar \partial}
\newcommand{\tensor}{\otimes}

\newcommand{\CC}{\mathbb{C}}
\newcommand{\RR}{\mathbb{R}}

\newcommand{\A}{\mathcal{A}}
\newcommand{\F}{\mathcal{F}}
\newcommand{\K}{\mathcal{K}}
\newcommand{\Q}{\mathcal{Q}}

\DeclareMathOperator{\ad}{ad}
\DeclareMathOperator{\At}{At}

\DeclareMathOperator{\End}{End}
\DeclareMathOperator{\Gal}{Gal}
\DeclareMathOperator{\Gr}{Gr}
\DeclareMathOperator{\Hom}{Hom}
\DeclareMathOperator{\id}{id}
\DeclareMathOperator{\pardeg}{par-deg}
\DeclareMathOperator{\ParEnd}{ParEnd}
\DeclareMathOperator{\parmu}{par-\mu}
\DeclareMathOperator{\Res}{Res}

\renewcommand{\leq}{\leqslant}

\begin{document}

\baselineskip=15pt

\title[Parabolic principal Higgs bundles]{Hermitian-Einstein
connections on polystable parabolic principal Higgs bundles}

\author[I. Biswas]{Indranil Biswas}

\address{School of Mathematics, Tata Institute of Fundamental
Research, Homi Bhabha Road, Bombay 400005, India}

\email{indranil@math.tifr.res.in}

\author[M. Stemmler]{Matthias Stemmler}

\email{stemmler@math.tifr.res.in}



\date{}

\begin{abstract}
Given a smooth complex projective variety $X$ and a smooth 
divisor $D$ on 
$X$, we prove the existence of Hermitian-Einstein connections, with 
respect to a Poincar\'e-type metric on $X \setminus D$, on polystable 
parabolic principal Higgs bundles with parabolic structure over $D$,
satisfying certain conditions on its restriction to $D$.
\end{abstract}

\maketitle

\section{Introduction}

The Hitchin-Kobayashi correspondence relating the stable vector bundles
and the solutions of the Hermitian-Einstein equation has turned out to
be extremely useful and important (see \cite{Do}, \cite{UY},
\cite{Si}). The Hitchin-Kobayashi correspondence has evolved into
a general principle finding generalizations to numerous contexts. 
Here we consider the parabolic Higgs $G$-bundles from this point
of view.

Parabolic vector bundles on curves were introduced by 
Seshadri \cite{Se77}. This was generalized to higher dimensional 
varieties by Maruyama and Yokogawa \cite{MY92}. Motivated by the 
characterization of 
principal bundles using Tannakian category theory given by Nori 
\cite{No76}, in \cite{BBN01}, parabolic principal bundles were defined.
Later ramified principal bundles were defined in
\cite{BBN03}; it turned out that there is a natural bijective 
correspondence between ramified principal bundles and parabolic 
principal bundles, cf.\ \cite{BBN03}, \cite{Bi06}. Higgs 
fields on ramified principal bundles were defined in \cite{Bi08}.

In \cite{Bq97}, Biquard considered vector bundles on a compact K\"ahler 
manifold $(X, \omega_0)$, with parabolic structure over a smooth divisor 
$D$, equipped with a Higgs field that has a logarithmic singularity on 
$D$. He showed that these data induce certain Higgs bundles (in an 
adapted sense) on $D$, which he calls ``sp\'ecialis\'es''. In the case 
of Higgs fields with nilpotent residue on $D$, these are just the graded
pieces of the parabolic filtration equipped with an induced Higgs
structure. Given a stable parabolic Higgs 
bundle such that these induced bundles are polystable and satisfy an 
additional condition on their slope, he proves the existence of a 
Hermitian-Einstein metric on $X \setminus D$ with respect to a 
Poincar\'e-type K\"ahler metric. The Hermitian-Einstein metric is
unique up to multiplication by a constant element of ${\mathbb R}^+$.

Our aim here is to extend Biquard's result to the case of 
parabolic principal Higgs $G$-bundles, where $G$ is a connected 
reductive linear algebraic group defined over $\CC$. Given such a bundle 
$(E_G, \theta)$, there is an adjoint parabolic Higgs vector bundle 
$(\ad(E_G), \ad(\theta))$. The Higgs field $\ad(\theta)$ has a 
nilpotent residue on 
$D$. This $\ad(\theta)$ induces Higgs fields on the
graded pieces $\Gr_\alpha \ad(E_G)$ for the parabolic vector
bundle $\ad(E_G)$. The Higgs field on $\Gr_\alpha \ad(E_G)$ 
induced by $\theta$ will be denoted by ${\rm ad}(\theta)_\alpha$.

Let $\psi\, :\, E_G\, \longrightarrow\, X$ be the natural projection.
The restriction of $\psi$ to $\psi^{-1}(D)$ will be denoted by
$\widehat\psi$. Let ${\mathcal K}$ be the trivial vector bundle
over $\psi^{-1}(D)$ with fiber $\text{Lie}(G)$. The group 
$G$ acts on ${\mathcal K}$ using the adjoint action of $G$
on $\text{Lie}(G)$. Define the invariant direct image
$$
{\mathcal E}\, :=\, (\widehat{\psi}_*{\mathcal K})^G\, ,
$$
which is a vector bundle over $D$. The Higgs field $\theta$ defines
a Higgs field on $\mathcal E$, which will be denoted by $\theta'$.

Fix a K\"ahler form $\omega_0$ on $X$ such that the corresponding
class in $H^2(X,\, {\mathbb R})$ is integral.

We obtain the following (see Theorem \ref{result} and 
Proposition \ref{converse}):

\begin{theorem}
Let $(E_G, \theta)$ be a parabolic Higgs $G$-bundle on $X$ 
such that $(E_G, \theta)$ is polystable with respect to
$\omega_0$, and satisfies the following two conditions:
\begin{itemize}
\item The Higgs bundle $({\mathcal E}\, ,\theta')$
on $D$ is polystable, and

\item for the graded pieces $(\Gr_\alpha \ad(E_G)\, , {\rm 
ad}(\theta)_\alpha)$ of $(\ad(E_G)\vert_D\, ,{\rm ad}(\theta)\vert_D)$
the condition
\begin{equation*}
  \mu(\Gr_\alpha \ad(E_G)) = - \alpha \deg(N)
\end{equation*}
holds, where degrees are computed using $\omega_0$ and $N$ is the 
normal bundle to $D$.
\end{itemize}
Then there is a Hermitian-Einstein connection on $E_G$ over
$X \setminus D$ with respect to the Poincar\'e-type metric.

Conversely, if there is such a Hermitian-Einstein connection 
satisfying the condition that
the induced connection on the adjoint vector bundle $\ad(E_G)|_{X \setminus D}$
lies in the space $\A$ (see \eqref{space of connections}), then
$(E_G, \theta)$ is polystable with respect to $\omega_0$.
\end{theorem}

\section{Parabolic Higgs bundles}\label{sec2}

Let $X$ be a connected smooth complex projective variety of complex 
dimension 
$n$, and let $D$ be a smooth reduced effective divisor on $X$. We first 
recall the definition of a parabolic Higgs vector bundle on $X$ with
parabolic structure over $D$.

A \textit{parabolic vector bundle} $E_\ast$ on $X$ with parabolic 
divisor $D$ is a holomorphic vector bundle $E$ on $X$ together with a 
parabolic structure on it, which is given by a decreasing filtration 
$\{F_\alpha(E)\}_{0 \leq \alpha \leq 1}$ of holomorphic 
subbundles of the restriction $E\vert_D$, which is
continuous from the left, satisfying the conditions that 
$F_0(E) \,=\, E\vert_D$ and $F_1(E) \,=\, 0$. The parabolic weights of 
$E_\ast$ are the numbers $0 \leq \alpha_1 < \cdots < \alpha_l < 1$ such
that $F_{\alpha_i + \varepsilon}(E) \neq F_{\alpha_i}(E)$ for all
$\varepsilon \,>\, 0$. For later use, we denote the graded pieces of 
this filtration as
\begin{equation}\label{egr}
  \Gr_\alpha E \,:= \, F_\alpha(E)/F_{\alpha + \varepsilon}(E), \quad 
\alpha 
\in \{ \alpha_1, \ldots, \alpha_l \}, \quad \varepsilon > 0 \,
\text{~sufficiently~small}.
\end{equation}
Let $\ParEnd(E_\ast)$ be the sheaf of holomorphic sections of 
$\End(E)\,=\, E\otimes E^\ast$ 
which preserve the above filtration of $E\vert_D$. Let $\Omega_X^k(\log 
D)$ be the vector bundle on $X$ defined by the sheaf of logarithmic 
$k$-forms. Note that there is a residue homomorphism
\[
  \Res_D: \ParEnd(E_\ast) \tensor \Omega_X^1(\log D)\,\longrightarrow\,
\ParEnd(E_\ast)|_D
\]
defined by the natural residue homomorphism
$\Omega_X^1(\log D)\,\longrightarrow\, {\mathcal O}_D$.

\begin{definition}
A {\em parabolic Higgs vector bundle} with parabolic divisor $D$ is a 
pair $(E_\ast, \theta)$ consisting of a parabolic vector bundle $E_\ast$ on $X$ with parabolic divisor $D$ and a section
\[
  \theta \in H^0(X, \ParEnd(E_\ast) \tensor \Omega_X^1(\log D))\, ,
\]
called the {\em Higgs field}, such that the following two conditions 
are satisfied:
\begin{itemize}
\item $\theta \bigwedge \theta \in H^0(X, \ParEnd(E_\ast) \tensor 
\Omega_X^2(\log D))$ vanishes identically, where the multiplication is 
defined using the Lie algebra structure of the fibers of $\End(E)$, and
the exterior product $\Omega_X^1(\log D)\otimes 
\Omega_X^1(\log D)\, \longrightarrow\, \Omega_X^2(\log D)$, and

\item the residue $\Res_D(\theta)$ is nilpotent with respect to the parabolic filtration in the sense that
\[
  \Res_D(\theta)(F_\alpha(E)) \subset F_{\alpha + \varepsilon}(E)
\]
for some $\varepsilon > 0$.
\end{itemize}
\end{definition}

In the following, we will omit the subscript ``$\ast$'' in $E_\ast$,
and denote a 
parabolic vector bundle by the same symbol as its underlying bundle.

Now we will recall the definitions of ramified Higgs
principal bundles and parabolic Higgs 
principal bundles. For this, let $G$ be a connected reductive
linear algebraic group defined over $\CC$.

\begin{definition}
A {\em ramified $G$-bundle\/} over $X$ with ramification over $D$ is a smooth complex variety $E_G$ equipped with an algebraic right action of $G$
\[
  f: E_G \times G \longrightarrow E_G
\]
and a surjective algebraic map
\[
  \psi\,:\, E_G \,\longrightarrow \,X\, ,
\]
such that the following conditions are satisfied:
\begin{itemize}
\item $\psi \circ f = \psi \circ p_1$, where $p_1: E_G \times G \to 
E_G$ denotes the natural projection,

\item for each point $x \in X$, the action of $G$ on the reduced fiber 
$\psi^{-1}(x)_{\text{red}}$ is transitive,

\item the restriction of $\psi$ to $\psi^{-1}(X \setminus D)$ makes $\psi^{-1}(X \setminus D)$ a principal $G$-bundle over $X \setminus D$,
  meaning the map $\psi$ is smooth over $\psi^{-1}(X \setminus D)$ 
and the map to the fiber product
  \[
    \psi^{-1}(X \setminus D) \times G \longrightarrow \psi^{-1}(X \setminus D) \times_{X \setminus D} \psi^{-1}(X \setminus D)
  \]
  given by $(z, g) \longmapsto \big(z, f(z, g)\big)$ is an isomorphism,
\item the reduced inverse image $\psi^{-1}(D)_{\text{red}}$ is a smooth divisor on $E_G$, and
\item for each point $z \in \psi^{-1}(D)_{\text{red}}$, the isotropy 
group $G_z \subset G$ for the action of $G$ on $E_G$ is a finite cyclic 
group acting faithfully on the quotient line $T_z E_G/T_z 
(\psi^{-1}(D)_{\text{red}})$.
\end{itemize}
\end{definition}

Parabolic principal $G$-bundles were defined in \cite{BBN01} as 
functors from the category of rational $G$-representations to the 
category of parabolic vector bundles, satisfying certain conditions; 
this definition was modeled on \cite{No76}. 
There is a natural bijective correspondence between the ramified 
principal $G$-bundles with ramification over $D$ and parabolic principal 
$G$-bundles on $X$ with $D$ as the parabolic divisor (\cite{BBN03}, 
\cite{Bi06}). Let us briefly recall a construction of
parabolic principal $G$-bundles from ramified principal $G$-bundles.

Let $E_G$ be a ramified $G$-bundle on 
$X$ with ramification over $D$. There is a finite (ramified) Galois covering
\[
  \eta\,:\, Y \,\longrightarrow \, X
\]
such that the normalizer
\begin{equation}\label{en}
F_G \,:=\, \widetilde{E_G \times_X Y}
\end{equation}
of the 
fiber product $E_G \times_X Y$ is smooth. Write $\Gamma 
\,:=\,\Gal(\eta)$ 
for the Galois group of $\eta$. Let
\begin{equation}\label{h}
h\, :\, \Gamma\, \longrightarrow\, \text{Aut}(Y)
\end{equation}
be the homomorphism giving the action of $\Gamma$ on $Y$.
The projection $F_G\,\longrightarrow \, 
Y$ yields a $\Gamma$-linearized principal $G$-bundle on $Y$ in the 
following sense:

\begin{definition}
A {\em $\Gamma$-linearized principal $G$-bundle\/} on $Y$ is a 
principal $G$-bundle
\[
  \psi\,:\, F'_G \,\longrightarrow \,Y
\]
together with a left action of $\Gamma$ on $F'_G$
\[
  \rho\,:\, \Gamma \times F'_G \,\longrightarrow \,F'_G
\]
such that the following two conditions are satisfied:
\begin{itemize}
\item The actions of $\Gamma$ and $G$ on $F'_G$ commute, and
\item $\psi(\rho(\gamma, z)) = h(\gamma)(\psi(z))$ for all $(\gamma, z) 
\in \Gamma \times F'_G$, and $h$ is defined in \eqref{h}.
\end{itemize}
\end{definition}

Consider $F_G$ constructed in \eqref{en}.
Given a finite-dimensional complex $G$-module $V$, there is the 
associated $\Gamma$-linearized vector bundle $F_G(V)\,=\, F_G\times^G
V$ on $Y$ with fibers
isomorphic to $V$. This $F_G(V)$ in turn corresponds to a 
parabolic vector bundle 
on $X$ with $D$ as the parabolic divisor, cf.\ \cite{Bi97b}; this
parabolic vector bundle will be denoted by $E_G(V)$.

The earlier mentioned functor, from the category of rational 
$G$-representations to the category of parabolic vector bundles,
associated to the ramified $G$-bundle $E_G$ sends any $G$-module $V$
to the parabolic vector bundle $E_G(V)$ constructed above.

In the following, we will identify the notions of parabolic and 
ramified $G$-bundles.

Let $\mathfrak{g}$ be the Lie algebra of $G$; it is equipped with the 
adjoint action of $G$. Setting $V\,=\, \mathfrak g$, the parabolic 
vector bundle $E_G(\mathfrak{g})$ constructed as above is called 
the adjoint parabolic vector bundle of $E_G$, and it is denoted 
by $\ad(E_G)$.

Let $E_G$ be a ramified $G$-bundle over $X$ with ramification over $D$. 
Let
\begin{equation}\label{cK}
\K \,\subset\, TE_G
\end{equation}
be the holomorphic subbundle defined by the tangent space of the orbits 
of the 
action of $G$ on $E_G$; since all the isotropies, for the action of $G$ 
on $E_G$, are finite groups, $\K$ is indeed a subbundle. Note that $\K$ 
is identified with the trivial 
vector bundle over $E_G$ with fiber $\mathfrak{g}$. Let
$$
\Q\, :=\, TE_G/\K
$$
be the quotient vector bundle. The action of $G$ on $E_G$ 
induces an action of $G$ on the tangent bundle $TE_G$, which preserves 
the subbundle $\K$. Therefore, there is an induced action of $G$ on the
quotient bundle $\Q$. These actions in turn induce a linear action of 
$G$ on $H^0(E_G, \K \tensor \Q^\ast)$.
Combining the exterior algebra structure of $\Lambda \Q^\ast$ and the Lie algebra structure on the fibers of $\K = E_G \times \mathfrak{g}$, one obtains a homomorphism
\[
 \tau\,:\, (\K \tensor \Q^\ast) \tensor (\K \tensor \Q^\ast) 
\,\longrightarrow \, \K \tensor \Lambda^2 \Q^\ast\, .
\]
For $y\, \in\, E_G$, and $a\, , b\, \in\, (\K \tensor \Q^\ast)_y$, the
image $\tau(a\otimes b)$ will also be denoted by $a \bigwedge b$.
\begin{definition}\label{de5}
\mbox{}
\begin{enumerate}
\item A {\em Higgs field\/} on $E_G$ is a section
\[
 \theta \,\in\, H^0(E_G, \K \tensor \Q^\ast)
\]
such that
\begin{itemize}
\item $\theta$ is invariant under the action of $G$ on $H^0(E_G, \K \tensor \Q^\ast)$, and
\item $\theta \bigwedge \theta = 0$.
\end{itemize}
\item A {\em parabolic Higgs $G$-bundle\/} is a pair $(E_G, \theta)$ consisting of a parabolic $G$-bundle $E_G$ and a Higgs field $\theta$ on $E_G$.
\end{enumerate}
\end{definition}

Now let $H\,\subset\, G$ be a Zariski closed subgroup, and let 
$U\,\subset\, 
X$ be a Zariski open subset. The inverse image $\psi^{-1}(U) \subset 
E_G$ will be denoted by $E_G|_U$; as before, $\psi$ is the projection
of $E_G$ to $X$.

\begin{definition}
A {\em reduction of structure group of $E_G$ to $H$ over $U$\/} is a subvariety
\[
  E_H \,\subset\, E_G|_U
\]
satisfying the following conditions:
\begin{itemize}
\item $E_H$ is preserved by the action of $H$ on $E_G$,
\item for each point $x \in U$, the action of $H$ on $\psi^{-1}(x) \bigcap E_H$ is transitive, and
\item for each point $z \in E_H$, the isotropy subgroup $G_z$, for the 
action of $G$ on $E_G$, is contained in $H$. 
\end{itemize}
\end{definition}

Clearly, such an $E_H$ is a ramified $H$-bundle over $U$. Let
\begin{equation} \label{reduction}
  \iota\,:\, E_H \,\longrightarrow\, E_G|_U
\end{equation}
be a reduction of structure group of $E_G$ to $H$ over $U$. Define the 
bundles $\K_H$ and $\Q_H$ as before with respect to $E_H$ (in place of 
$E_G$). Then by \cite[(3.8)]{Bi08}, 
\[
  \Hom(\Q_H, \K_H) \,\subset\, \iota^\ast \Hom(\Q, \K)\, .
\]
Let $\theta \in H^0(E_G, \Hom(\Q, \K))$ be a Higgs field on $E_G$.

\begin{definition}
The reduction $E_H$ in \eqref{reduction} is said to be {\em compatible\/} with the Higgs field 
$\theta$ if
\[
  \theta|_{E_H} \in H^0(E_H, \Hom(\Q_H, \K_H)) \subset H^0(E_H, \iota^\ast \Hom(\Q, \K))\, .
\]
\end{definition}

Fix a very ample line bundle $\zeta$ on $X$. Define the degree $\deg 
\F$ (respectively, the parabolic degree 
$\pardeg E_\ast$) of a torsion-free coherent sheaf $\F$ (respectively,
a parabolic vector bundle $E_\ast$) on $X$ with respect to this
polarization $\zeta$.

Fix a basis of $H^0(X, \, \zeta)$. Using this basis we get an embedding 
of $X$ in ${\mathbb C}{\mathbb P}^{N-1}$, where $N\, =\, \dim
H^0(X, \, \zeta)$. Let $\omega_0$ be the restriction to $X$
of the Fubini-Study metric on ${\mathbb C}{\mathbb P}^{N-1}$.

Let $H$ be a parabolic subgroup of $G$. Then $G/H$ is a complete 
variety, and the quotient map $G \longrightarrow G/H$ defines a 
principal $H$-bundle over $G/H$. For any character $\chi$ of $H$, let
\[
  L_{\chi} \,\longrightarrow \, G/H
\]
be the line bundle associated to this principal $H$-bundle for the 
character $\chi$.
Let $R_u(H)$ be the unipotent radical of $H$ (it is the unique maximal
normal unipotent subgroup). The group $H/R_u(H)$ is called the
\textit{Levi quotient} of $H$. There are subgroups $L(H)\, 
\subset\, H$ such 
that the composition $L(H)\, \hookrightarrow\, H\,\longrightarrow\, 
H/R_u(H)$ is an isomorphism. Such a subgroup $L(H)$ is called a 
\textit{Levi subgroup} of $H$. Any two Levi subgroups of $H$ are 
conjugate by some element of $H$.

Let $Z_0(G) \,\subset\, G$ be the connected component, containing the 
identity element, of the center of $G$. It is known that $Z_0(G) \,
\subset\, H$. A character $\chi$ of $H$ which is trivial on $Z_0(G)$ is 
called {\em strictly antidominant\/} if the corresponding line bundle
$L_\chi$ over $G/H$ (defined above) is ample.

\begin{definition}
A parabolic Higgs $G$-bundle $(E_G, \theta)$ is called {\em stable\/} if 
for every quadruple $(H, \chi, U, E_H)$, where
\begin{itemize}
\item $H \subset G$ is a proper parabolic subgroup,
\item $\chi$ is a strictly antidominant character of $H$,
\item $U \subset X$ is a non-empty Zariski open subset such that the codimension of $X \setminus U$ is at least two, and
\item $E_H \subset E_G|_U$ is a reduction of structure group of $E_G$ to $H$ over $U$ compatible with~$\theta$,
\end{itemize}
the following holds:
\[
   \pardeg(E_H(\chi)) \,>\, 0\, ,
\]
where $E_H(\chi)$ is the parabolic line bundle over $U$ associated to 
the parabolic $H$-bundle $E_H$ for the one-dimensional representation 
$\chi$ of $H$.
\end{definition}

Let $E_G$ be a parabolic $G$-bundle over $X$. A reduction of structure 
group $E_H \subset E_G$ to some parabolic subgroup $H \subset G$ is 
called {\em admissible\/} if for each character $\chi$ of $H$ which 
is trivial on $Z_0(G)$, the associated parabolic line bundle 
$E_H(\chi)$ over $X$ satisfies the following condition:
\[
	\pardeg(E_H(\chi)) \,=\, 0\, .
\]

\begin{definition}\label{de-ps}
A parabolic Higgs $G$-bundle $(E_G, \theta)$ is called {\em 
polystable\/} if either $(E_G, \theta)$ is stable, or there is a proper 
parabolic subgroup $H \subset G$ and a reduction of structure group
\[
	E_{L(H)} \,\subset\, E_G
\]
of $E_G$ to a Levi subgroup $L(H)\, \subset\, H$ over $X$ such that the 
following conditions are satisfied:
\begin{itemize}
\item The reduction $E_{L(H)} \subset E_G$ is compatible with $\theta$,
\item the parabolic Higgs $L(H)$-bundle $(E_{L(H)}, \theta|_{E_{L(H)}})$ is stable (from the first condition it follows that $\theta|_{E_{L(H)}}$ is a Higgs field on $E_{L(H)}$), and
\item the reduction of structure group of $E_G$ to $H$, obtained by extending the structure group of $E_{L(H)}$ using the inclusion of $L(H)$ in $H$, is admissible.
\end{itemize}
\end{definition}

\section{Hermitian-Einstein connection on a parabolic Higgs 
$G$-bundle}\label{sec3}

Let $E_G$ be a parabolic $G$-bundle over $X$. Let
\begin{equation} \label{atiyah}
  0 \longrightarrow \ad(E_G) \longrightarrow \At(E_G) \longrightarrow TX \longrightarrow 0
\end{equation}
be the Atiyah exact sequence for the $G$-bundle $E_G$ over $X \setminus 
D$. Recall that a {\em complex connection\/} on $E_G$ over $X \setminus 
D$ is a $C^\infty$ splitting of this exact sequence. Fix a maximal 
compact subgroup $K \subset G$. A complex connection on $E_G$ over $X
\setminus D$ is called {\em unitary\/} if it is induced by a 
connection on a smooth reduction of structure group $E_K$ of $E_G$ to 
$K$ over $X \setminus D$. Note that \eqref{atiyah} is a short
exact sequence of sheaves of Lie algebras. For a complex unitary 
connection $\nabla$ on $E_G$ over $X \setminus D$, its {\em curvature 
form}
\[
  F\,\in\, H^0(X \setminus D,\, \Lambda^{1,1} T^\ast X \tensor \ad(E_G))
\]
measures the obstruction of the splitting of \eqref{atiyah} 
defining $\nabla$ to be Lie algebra structure preserving; see 
\cite{At} for the details.

For a parabolic Higgs $G$-bundle $(E_G, \theta)$ on $X$, its 
restriction to $X \setminus D$ is a Higgs $G$-bundle in the
usual sense. Given a smooth reduction of structure group $E_K$ of $E_G$ 
to a maximal compact subgroup $K \subset G$ over $X \setminus D$, the 
Cartan involution of $\mathfrak{g}$ with respect to $K$ induces an 
involution of the adjoint vector bundle $\ad(E_G)$ over $X \setminus 
D$; this involution of $\ad(E_G)$ will be denoted by $\phi$. 
Writing $\theta = \sum_i \theta_i dz^i$ in local holomorphic coordinates 
$z^1, \ldots, z^n$ on $X$ around a point $x \in X \setminus D$, define
\[
  \theta^\ast \,:=\, - \sum_i \phi(\theta_i) d\bar z^i\, .
\]
This definition is clearly independent of the choice of local 
coordinates.

Let $\mathfrak{z}$ be the center of the Lie algebra $\mathfrak{g}$ of 
$G$. Since the adjoint action of $G$ on $\mathfrak{z}$ is trivial,
an element $\lambda \in \mathfrak{z}$ defines a smooth 
section of $\ad(E_G)$ over $X \setminus D$, which will also be denoted 
by $\lambda$.

\begin{definition}\label{dhe}
Let $(E_G, \theta)$ be a parabolic Higgs $G$-bundle on $X$. A complex 
unitary connection on $E_G$ over $X \setminus D$ is called a {\em 
Hermitian-Einstein connection\/} with respect to a K\"ahler metric 
$\omega$ on $X \setminus D$ and the Higgs field $\theta$, if its 
curvature form $F$ satisfies the equation
\[
\Lambda_\omega(F + [\theta, \theta^\ast]) \,= \,\lambda
\]
for some $\lambda \in \mathfrak{z}$, where the operation $[{\cdot}, 
{\cdot}]$ is defined using the exterior product on forms and the Lie 
algebra structure of the fibers of $\ad(E_G)$.
\end{definition}

Note that $\lambda$ in Definition \ref{dhe} lies in
$\mathfrak{z}\bigcap \text{Lie}(K)$.

In \cite{Bq97}, Biquard introduces a Poincar\'e-type metric on $X 
\setminus D$ as follows: Let $\sigma$ be the canonical section of 
the line bundle ${\mathcal O}_X(D)$ on $X$ associated to the 
divisor $D$, meaning $D$ is the 
zero divisor of $\sigma$. Let $\omega_0$ be the K\"ahler form
on $X$ that we fixed earlier. Choose a Hermitian metric $||{\cdot}||$ on 
the fibers of ${\mathcal O}_X(D)$. Then
\begin{equation}\label{om}
  \omega \,:=\, T \omega_0 - \sqrt{-1} \partial \dbar \log \log^2 
||\sigma||^2
\end{equation}
defines a K\"ahler metric on $X \setminus D$ for $T\,\in\,
{\mathbb R}^+$ large enough.

In \cite{Bq97}, Biquard proves the existence of 
Hermitian-Einstein metrics on stable parabolic Higgs vector bundles 
under certain additional conditions (see \cite[Th\'eor\`eme 8.1]{Bq97}). 
In his definition of parabolic 
Higgs vector bundles he does not require the residue of the Higgs field 
to be nilpotent.

Let $(E\, , \theta)$ be a parabolic Higgs vector bundle.
Consider the graded pieces $\Gr_\alpha E$ in \eqref{egr}. Let
$$
\theta_{\alpha}\, :=\,\theta\vert_D\, :\,\Gr_\alpha E\,
\longrightarrow\, (\Gr_\alpha E)\otimes (\Omega^1_X(\log D)\vert_D)
$$
be the homomorphism given by $\theta$. Since the residue of $\theta$ is 
nilpotent with respect to the quasi-parabolic filtration of
$E\vert_D$, the composition
$$
\Gr_\alpha E\, \stackrel{\theta_\alpha}{\longrightarrow}\,
(\Gr_\alpha E)\otimes (\Omega^1_X(\log D)\vert_D)
\, \stackrel{{\rm Id}\otimes{\rm Res}}{\longrightarrow}\,
\Gr_\alpha E\otimes {\mathcal O}_D \,=\, \Gr_\alpha E
$$
vanishes identically. Therefore, $\theta_{\alpha}\, \in\,
H^0(D,\, \text{End}(\Gr_\alpha E)\otimes \Omega^1_D)$.
The integrability
condition $\theta\bigwedge\theta\,=\, 0$ immediately implies that
$\theta_\alpha\bigwedge\theta_\alpha\,=\, 0$. Therefore,
$(\Gr_\alpha E\, , \theta_\alpha)$ is a Higgs vector bundle on $D$.

In \cite[pp.\ 47--48]{Bq97}, Biquard uses the parabolic structure of $E$
to construct a background metric on $E$ over $X \setminus D$.
Let $\nabla$ be the corresponding Chern connection. He then restricts
his attention to connections lying in the space
\begin{equation} \label{space of connections}
  \A := \big\{ \nabla + a : a \in \widehat C_\delta^{1 + \vartheta}(\Omega_X^1 \tensor \End(E)) \big\}
\end{equation}
(see \cite[p.\ 58 and p.\ 70]{Bq97}), where the H\"older space
$\widehat C_\delta^{1 + \vartheta}(\Omega_X^1 \tensor \End(E))$ 
is defined in \cite[pp.\ 53--54]{Bq97}. Let
$$
N\, \longrightarrow\, D
$$
be the normal line bundle of the divisor $D$.

With these definitions, Biquard's theorem can be formulated as 
follows:

\begin{theorem}\label{biquard}
Let $(E, \theta)$ be a stable parabolic Higgs vector bundle on $X$ with 
parabolic divisor $D$. Assume that all
the graded Higgs bundles $(\Gr_\alpha E, \theta_\alpha)$ are polystable
and satisfy the condition
\begin{equation} \label{graded slope (biquard)}
  \mu(\Gr_\alpha E) = \parmu(E) - \alpha \deg(N)
\end{equation}
with respect to $\omega_0$. Then there is a Hermitian metric $h$ on $E$
over $X \setminus D$, with Chern connection in $\A$, which is 
Hermitian-Einstein with respect to the 
Poincar\'e-type metric $\omega$, meaning its Chern curvature form 
$F$ satisfies
\[
  \sqrt{-1} \Lambda_\omega(F + [\theta, \theta^\ast]) \,=\, 
\lambda\cdot\id_E
\]
for some $\lambda \in \RR$.

Such a Hermitian metric is unique up to a constant scalar multiple.
\end{theorem}

\section{Existence of Hermitian-Einstein connection}

Let $(E_G\, ,\theta)$ be a ramified Higgs $G$-bundle. Let
$$
\psi\, :\, E_G\,\longrightarrow\, X
$$
be the natural projection. The reduced divisor
$\psi^{-1}(D)_{\rm red}$ will be denoted by $\widetilde{D}$.
Let
$$
\widehat{\psi}\, :=\, \psi\vert_{\widetilde{D}}\, :\, \widetilde{D}
\, \longrightarrow\, D
$$
be the restriction. Consider the subbundle $\mathcal K$ defined
in \eqref{cK}. The action of the group $G$ on $\widetilde{D}$
produces an action of $G$ on the direct image
$\widehat{\psi}_*{\mathcal K}\, \longrightarrow\, D$. Define
the invariant part
\begin{equation}\label{cE}
{\mathcal E} \, :=\, (\widehat{\psi}_*{\mathcal K})^G
\, \longrightarrow\, D\, ;
\end{equation}
it is a vector bundle over $D$.

We will give an explicit description of the vector bundle $\mathcal E$.
As before, the isotropy subgroup of any $z\, \in\, 
\widetilde{D}$, for the action of $G$ on $\widetilde{D}$, will be 
denoted by $G_z$. Let
$$
{\mathfrak g}_z\, :=\, {\mathfrak g}^{G_z}\, \subset\,
{\mathfrak g}
$$
be the space of invariants for the adjoint action of $G_z$. This
${\mathfrak g}_z$ is clearly a subalgebra of ${\mathfrak g}$. The elements
of $G_z$ are semisimple because $G_z$ is a finite group. Since $G_z$
is cyclic, the Lie subalgebra ${\mathfrak g}_z$
is reductive (see \cite[p.\ 26, Theorem]{Hu}). Let
${\mathcal S}$ be the subbundle of the trivial vector bundle
$\widetilde{D}\times {\mathfrak g}\, \longrightarrow\, \widetilde{D}$
whose fiber over any $z\, \in\, \widetilde{D}$ is the subalgebra
${\mathfrak g}_z$. The action of $G$ on $\widetilde{D}$ and the adjoint
action of $G$ on $\mathfrak g$ combine together to define an 
action of $G$
on $\widetilde{D}\times {\mathfrak g}$; the identification
between ${\mathcal K}\vert_{\widetilde{D}}$ and 
$\widetilde{D}\times {\mathfrak g}$ commutes with the actions of 
$G$ . The action of $G$ on
$\widetilde{D}\times {\mathfrak g}$ clearly preserves the subbundle
$\mathcal S$. We have
\begin{equation}\label{q}
D\, =\, \widetilde{D}/G ~\, \text{~and~}\, ~
{\mathcal E} \, =\, {\mathcal S}/G\, .
\end{equation}
That ${\mathcal S}/G$ is a vector bundle over $\widetilde{D}/G$
follows from the fact that the isotropy subgroups act trivially
on the fibers of ${\mathcal S}$.

Let $h$ be any $G$-invariant nondegenerate symmetric bilinear form on
$\mathfrak g$. The restriction of $h$ to the centralizer, in $\mathfrak 
g$, of any semisimple element of $G$ is known to be nondegenerate. From 
this it follows that the bilinear form induced by $h$ on the vector 
bundle $\mathcal S$ in \eqref{q} is nondegenerate. Since $h$ 
is $G$-invariant, from \eqref{q} we 
conclude that this nondegenerate bilinear form on $\mathcal S$ descends 
to a nondegenerate bilinear form on ${\mathcal E}$. This implies
that ${\mathcal E}^\ast\, =\, {\mathcal E}$, in particular, $\deg
({\mathcal E})\,=\,0$ with respect to any polarization on $D$.

Recall that
the fibers of $\mathcal K$ are identified with $\mathfrak g$.
Using this Lie algebra structure of the fibers of $\mathcal K$,
the Higgs field $\theta$ defines a homomorphism
$$
{\mathcal K}\, \longrightarrow\, {\mathcal K}\otimes {\mathcal 
Q}^\ast\, \longrightarrow\, 0
$$
of vector bundles. On the other hand, over $\widetilde D$,
we have a natural restriction homomorphism
$$
{\mathcal Q}^\ast\vert_{\widetilde D}\, \longrightarrow\, 
\Omega^1_{\widetilde D}
$$
of vector bundles. Combining these two homomorphisms, we have a
homomorphism of vector bundles
$$
\beta\, :\, {\mathcal K}\vert_{\widetilde D}\, \longrightarrow\, 
{\mathcal K}\vert_{\widetilde D}\otimes \Omega^1_{\widetilde D}\, .
$$
The group $G$ acts on both ${\mathcal K}\vert_{\widetilde D}$
and $\Omega^1_{\widetilde D}$. The above homomorphism $\beta$
commutes with the actions of $G$. Therefore, $\beta$ produces
a homomorphism
\begin{equation}\label{thp}
\theta'\, :\, {\mathcal E} \, =\, (\widehat{\psi}_*{\mathcal K})^G
\, \longrightarrow\, ({\mathcal K}\vert_{\widetilde D}\otimes 
\Omega^1_{\widetilde D})^G \,=\, {\mathcal E}\otimes\Omega^1_D\, ,
\end{equation}
where $\mathcal E$ is defined in \eqref{cE}. From
the condition $\theta \bigwedge \theta = 0$ (see Definition
\ref{de5}) it follows that $\theta'$ is a Higgs field on the
vector bundle ${\mathcal E}$.

Consider the adjoint parabolic vector bundle $\text{ad}(E_G)$ for the 
ramified $G$-bundle $E_G$. The Higgs field $\theta$ produces a Higgs
field on the parabolic vector bundle $\text{ad}(E_G)$. This
induced Higgs field on $\text{ad}(E_G)$ will be denoted by
$\text{ad}(\theta)$.

\begin{theorem} \label{result}
Let $(E_G, \theta)$ be a parabolic Higgs $G$-bundle on $X$ 
such that $(E_G, \theta)$ is polystable with respect to the
K\"ahler form $\omega_0$ (see \eqref{om}), and satisfies the following 
two conditions:
\begin{itemize}
\item The Higgs bundle $({\mathcal E}\, ,\theta')$ constructed
in \eqref{cE} and \eqref{thp} is polystable, and

\item for the graded pieces $(\Gr_\alpha \ad(E_G)\, , {\rm 
ad}(\theta)_\alpha)$ of $(\ad(E_G)\vert_D\, ,{\rm ad}(\theta)\vert_D)$,
the condition
\begin{equation} \label{graded slope}
  \mu(\Gr_\alpha \ad(E_G)) = - \alpha \deg(N)
\end{equation}
holds, where degrees are computed using $\omega_0$ and $N$
is the normal bundle of $D$.
\end{itemize}
Then there is a Hermitian-Einstein connection on $E_G$ over
$X \setminus D$ with respect to the Poincar\'e-type metric
described in Section \ref{sec3}.
\end{theorem}

\begin{proof}
We first note that it is enough to prove the theorem
under the stronger assumption that the parabolic Higgs $G$-bundle
$(E_G\, ,\theta)$ is stable. Indeed, a polystable
parabolic Higgs $G$-bundle $(E_G\, ,\theta)$ admits a reduction 
of structure group $E_{L(P)}\, \subset\, E_G$ to a Levi subgroup 
$L(P)$ of some parabolic subgroup $P$ of $G$
such that the corresponding parabolic Higgs $L(P)$-bundle
$(E_{L(P)}\, ,\theta)$ is stable (see Definition \ref{de-ps}). 
The connection on $E_G$ induced by a Hermitian-Einstein 
connection on $E_{L(P)}$ is again Hermitian-Einstein. Hence
it suffices to prove the theorem for $(E_G\, ,\theta)$ stable.

Henceforth, in the proof we assume that $(E_G\, ,\theta)$ is stable.

We will now show that it is enough to prove the theorem
under the assumption that $G$ is semisimple.

As before, $Z_0(G)\, \subset\, G$ is the connected component,
containing the identity element, of the center of $G$. The
normal subgroup $[G\, ,G]\, \subset\, G$ is semisimple, because
$G$ is reductive. We have natural homomorphisms
$$
Z_0(G)\times [G\, ,G]\, \longrightarrow\, G \, \longrightarrow
\, (G/Z_0(G))\times (G/[G\, ,G])\, .
$$
Both the homomorphisms are surjective with finite kernel. In
particular, both the homomorphisms of Lie algebras are 
isomorphisms. Let $\rho\, :\, A\, \longrightarrow\, B$ be
a homomorphism of Lie groups such that the corresponding 
homomorphism of Lie algebras is an isomorphism, let $E_A$
be a principal $A$-bundle, and let $E_B\, :=\, E_A\times^\rho B$
be the principal $B$-bundle obtained by extending the structure
group of $E_A$ using $\rho$. Then there is a natural bijective
correspondence between the connections on $E_A$ and the 
connections on $E_B$. The curvature of a connection on $E_B$
is given by the curvature of the corresponding connection
on $E_A$ using the homomorphism of Lie algebras associated
to $\rho$. Therefore, to prove the theorem for $G$, it 
is enough to prove it for $G/Z_0(G)$ and $G/[G\, ,G]$ separately. 
But $G/[G\, ,G]$ is a product of copies of ${\mathbb C}^*$, hence
in this case the theorem follows immediately from Theorem 
\ref{biquard}. The group $G/Z_0(G)$ is semisimple. Hence it is
enough to prove the theorem under the assumption that $G$ is
semisimple.

Henceforth, in the proof we assume that $G$ is semisimple.

Denote by $\eta: Y \longrightarrow X$ the Galois covering with Galois 
group $\Gamma := \Gal(\eta)$ and by $F_G$ the $\Gamma$-linearized 
$G$-bundle on $Y$ corresponding to $E_G$ as described in Section 
\ref{sec2} . According 
to \cite[Proposition 4.1]{Bi08}, the Higgs field $\theta$ on $E_G$ 
corresponds to a $\Gamma$-invariant Higgs field $\widetilde \theta$ on 
$F_G$. This induces a $\Gamma$-invariant Higgs field $\ad(\widetilde 
\theta)$ on the $\Gamma$-linearized vector bundle $\ad(F_G)$. By 
\cite[Theorem 5.5]{Bi97a}, this in turn corresponds to a Higgs field 
$\ad(\theta)$ on the parabolic vector bundle $\ad(E_G)$. This way 
we construct
the parabolic Higgs vector bundle $(\ad(E_G), \ad(\theta))$ on $X$ 
defined earlier.

The strategy of the proof is to show that the hypotheses of Biquard's 
Theorem \ref{biquard} are satisfied for $(\ad(E_G), \ad(\theta))$ 
and that the 
resulting Hermitian-Einstein connection on $\ad(E_G)|_{X \setminus D}$ is induced by a Hermitian-Einstein connection on $E_G|_{X \setminus D}$.

First we show that $(\ad(E_G), \ad(\theta))$ is parabolic polystable. 
Since $(E_G, \theta)$ is stable by hypothesis, it follows as in 
\cite[Lemma 4.2]{Bi08} that $(F_G, \widetilde{\theta})$ is 
$\Gamma$-stable. In \cite{AB01} it was shown that if
a principal $G$-bundle $E^1_G$ is stable, then its adjoint
vector bundle $\text{ad}(E^1_G)$ is polystable (see
\cite[p.\ 212, Theorem 2.6]{AB01}). The proof in \cite{AB01}
goes through if $(E^1_G, \theta^1)$ is $\Gamma$-stable, and gives 
that $(\text{ad}(E^1_G), \text{ad}(\theta^1))$ is 
$\Gamma$-polystable. 
Since the proof goes through verbatim with obvious modifications
due to the Higgs field, we refrain from repeating the proof.
Therefore, we have $(\ad(F_G), \text{ad}(\widetilde{\theta}))$ to 
be $\Gamma$-polystable.

Since $(\ad(F_G)\, , \text{ad}(\widetilde{\theta}))$ is 
$\Gamma$-polystable, the parabolic Higgs vector bundle
$$(\ad(E_G)\, ,\ad(\theta))$$ is parabolic polystable
(see \cite[p.\ 611, Theorem 5.5]{Bi97a}).

Let $M$ be a reductive complex linear algebraic group. The
connected component, containing the identity element, of the center of 
$M$ will be denoted by $Z_0(M)$. Let $(E_M\, ,
\theta_M)$ be a polystable principal Higgs $M$-bundle on a 
connected complex projective manifold. If $V$ is a complex $M$-module
such that $Z_0(M)$ acts on $V$ as scalar multiplications through a
character of $Z_0(M)$, then it is known that the associated Higgs vector 
bundle $(E_M\times^M V\, ,\theta_V)$ is polystable,
where $\theta_V$ is the Higgs field on the associated vector 
bundle $E_M\times^M V$ defined by
$\theta_M$. Indeed, this follows immediately from the fact that
$(E_M\, , \theta_M)$ has a Hermitian-Einstein connection; note 
that the connection on $(E_M\times^M V\, ,\theta_V)$ induced by
a Hermitian-Einstein connection on $(E_M\, , \theta_M)$ is also
Hermitian-Einstein, provided the above condition for the action
of $Z_0(M)$ on $V$ holds. (See \cite[p.\ 227, Theorem 4.10]{AB01}
for the Hermitian-Einstein connection on $(E_M\, , \theta_M)$.)

Since the Higgs bundle $({\mathcal E}\, ,\theta')$ constructed
in \eqref{cE} and \eqref{thp} is given to be polystable, from the
above observation it follows that each of the graded pieces 
$(\Gr_\alpha \ad(E_G), \ad(\theta)_\alpha)$ of $(\ad(E_G)\, 
,\ad(\theta))$ is polystable.

Since $G$ is semisimple, the Killing form on its Lie algebra 
$\mathfrak{g}$ is nondegenerate and thus induces an isomorphism 
$\ad(F_G) \simeq \ad(F_G)^\ast$. This implies that 
$\deg(\ad(F_G)) = 0$. By \cite[p.\ 318, (3.12)]{Bi97b}, we have
\[
  \# \Gamma \cdot \pardeg(\ad(E_G)) = \deg(\ad(F_G))\, ,
\]
and thus $\pardeg(\ad(E_G))\,=\,0$, or equivalently, 
$\parmu(\ad(E_G)) \,=\, 0$. Consequently, the hypothesis 
\eqref{graded slope} on the slopes of the graded pieces implies 
that the condition \eqref{graded slope (biquard)} in Theorem 
\ref{biquard} holds for the bundle $\ad(E_G)$. Therefore, 
we obtain from Theorem \ref{biquard} a Hermitian-Einstein metric 
on $\ad(E_G)$ over $X\setminus D$ with respect to the 
Poincar\'e-type metric.

Finally we have to show that the corresponding 
Hermitian-Einstein connection on $\ad(E_G)$ is induced by a 
connection on the 
principal Higgs $G$-bundle $E_G|_{X \setminus 
D}$; we note that if $\nabla$ is a connection on $E_G|_{X 
\setminus D}$ inducing the Hermitian-Einstein connection on 
$\ad(E_G)$, then $\nabla$ is automatically Hermitian-Einstein.

Let
$$
\Phi\, \in\, H^0(X \setminus D,\, (\ad(E_G) \tensor 
\ad(E_G))^\ast \tensor \ad(E_G))
$$
be the section defining the Lie bracket operation on $\ad(E_G)$.
It can be shown that a connection $\nabla_{\rm ad}$ on
$\ad(E_G)\vert_{X \setminus D}$ is induced by a connection
of $E_G|_{X \setminus D}$ if and only if $\Phi$ is parallel
with respect to the connection on $(\ad(E_G) \tensor
\ad(E_G))^\ast \tensor \ad(E_G)$ induced by $\nabla_{\rm ad}$.
Indeed, this follows from the fact that $G$ being semisimple 
the Lie algebra of the group of Lie algebra preserving 
automorphisms of $\mathfrak g$ coincides with $\mathfrak g$ (see 
proof of Theorem 3.7 of \cite{AB01}).

Therefore, to complete the proof of the theorem it suffices
to show that $\Phi$ is parallel with respect to the connection
on  $(\ad(E_G) \tensor\ad(E_G))^\ast \tensor \ad(E_G)$ induced 
by a Hermitian-Einstein connection on $\ad(E_G)$.

Since $\pardeg(\ad(E_G)) \,=\, 0$, it follows that
$$
\pardeg((\ad(E_G) \tensor \ad(E_G))^\ast \tensor \ad(E_G)) \,=\, 
0\, .
$$
The connection on $(\ad(E_G) 
\tensor \ad(E_G))^\ast \tensor \ad(E_G)$ induced by
the Hermitian-Einstein connection on $\ad(E_G)$ is also a 
Hermitian-Einstein connection. Since the Higgs field 
$\ad(\theta)$ is induced by the Higgs field $\theta$ on $E_G$, it 
follows that $\Phi$ is annihilated by the induced Higgs field on 
$(\ad(E_G) \tensor\ad(E_G))^\ast \tensor \ad(E_G)$. Thus the 
proof of Theorem \ref{result} is completed by Lemma \ref{lem1}.
\end{proof}

\begin{lemma}\label{lem1}
Let $(E\, , \theta)$ be a parabolic Higgs vector bundle on $X \setminus D$ 
admitting a Hermitian-Einstein connection $\nabla$ with respect to the 
Poincar\'e-type metric. Assume that $(E, \theta)$ is polystable, and
$\pardeg E\,=\, 0$. Let $s$ be a holomorphic section of $E$ such that 
$\theta(s) \,=\, 0$. Then $s$ is parallel with respect to $\nabla$.
\end{lemma}

\begin{proof}
Fix a Galois covering $\eta\, :\, Y\, \longrightarrow\, X$ such that
there is a $\Gamma$-linearized Higgs vector bundle $(V\, ,\varphi)$ on
$Y$ that corresponds to $(E, \theta)$, where $\Gamma\,=\, \text{Gal}(\eta)$.
Fix the polarization $\eta^*\zeta$ on $Y$, where $\zeta$ is the polarization
on $X$.

We know that $(V\, ,\varphi)$ is $\Gamma$-polystable because $(E, 
\theta)$ is polystable. Therefore, $(V\, ,\varphi)$ admits 
a Hermitian-Einstein connection \cite[p.\ 978, Theorem 1]{Si88}.

Let $\widetilde{s}$ be the holomorphic section of $V$ 
over $Y$ given by $\theta$. We
note that $\varphi( \widetilde{s})\,=\, 0$ because
$\theta(s) \,=\, 0$.
We have $\deg V\,=\, 0$ because $\pardeg E\,=\, 0$ \cite[p.\ 318, 
(3.12)]{Bi97b}. Since $(V\, ,\varphi)$ admits 
a Hermitian-Einstein connection with $\deg V\,=\, 0$,
and $\varphi( \widetilde{s})\,=\, 0$, it follows that the
holomorphic section $\widetilde{s}$ is flat with respect to the
Hermitian-Einstein connection on $(V\, ,\varphi)$ \cite[p.\ 548,
Lemma 3.4]{BS}.

If $s$ vanishes identically, then the lemma is obvious.
Assume that $s$ does not vanish identically. Since
$\widetilde{s}$ is flat with respect to the
Hermitian-Einstein connection on $(V\, ,\varphi)$, the
section $\widetilde{s}$ does not vanish at any point of $Y$.
Let $L^{\widetilde{s}}\, \subset\, V$ be the holomorphic line 
subbundle generated by $\widetilde{s}$. The action of $\Gamma$ on 
$V$ clearly preserves $L^{\widetilde{s}}$. Since
$(V\, ,\varphi)$ is $\Gamma$-polystable, this implies that
there is  a $\Gamma$-polystable Higgs vector bundle
$(V'\, ,\varphi')$ such that
$$
(V\, ,\varphi)\,=\, (V'\, ,\varphi')\, \oplus (L^{\widetilde{s}}\, ,0)
$$
as $\Gamma$-linearized Higgs vector bundles.

The above decomposition of the $\Gamma$-linearized Higgs vector bundle $(V\, 
,\varphi)$ produces a decomposition
$$
(E\, ,\theta)\,=\, (E'\, ,\theta')\, \oplus (L^s\, ,0)
$$
of the parabolic Higgs vector bundle; the line subbundle $L^s$ of $E$ is
generated by $s$.

The direct sum of the Hermitian-Einstein connections on $(E'\, ,\theta')$ and
$(L^s\, ,0)$ is a Hermitian-Einstein connection on $(E\, ,\theta)$.
Therefore, from the uniqueness of the Her\-mitian-Einstein 
connection (see the second part of
Theorem \ref{biquard}) it follows immediately that 
$s$ is parallel with respect to the Hermitian-Einstein 
connection $\nabla$.
\end{proof}

There is also a converse to Theorem \ref{result}:

\begin{proposition} \label{converse}
Let $(E_G, \theta)$ be a parabolic Higgs $G$-bundle on $X$. Suppose there is 
a Hermitian-Einstein connection on $E_G$ over $X \setminus D$ with respect to 
the Poincar\'e-type metric $\omega$ such that the induced connection on the 
adjoint vector bundle $\ad(E_G)|_{X \setminus D}$ lies in the space $\A$ (see 
\eqref{space of connections}). Then $(E_G, \theta)$ is polystable with 
respect to $\omega_0$.
\end{proposition}

\begin{proof}
By \cite[Proposition 7.2]{Bq97} we know that the parabolic 
degree of a parabolic sheaf on $X$ with respect to $\omega_0$ coincides with 
the degree of its restriction to $X \setminus D$ with respect to $\omega$, 
computed using a Hermitian metric with Chern connection in $\A$. Thus the 
proof in \cite[pp.\ 28--29]{RS88} of the proposition for ordinary 
principal bundles generalizes to our situation of parabolic Higgs 
$G$-bundles.
\end{proof}


\begin{thebibliography}{[BBN01]}

\bibitem[AB01]{AB01} {\scshape B.~Anchouche and I.~Biswas}: {\em 
Einstein-Hermitian connections on polystable principal bundles over a 
compact K\"ahler manifold}, Am.\ Jour.\ Math.\ {\bfseries 123}, 
207--228 (2001).

\bibitem[At57]{At} {\scshape M. F. Atiyah}: {\em Complex analytic
connections in fibre bundles},
Trans. Amer. Math. Soc. \textbf{85}, 181--207 (1957).

\bibitem[BBN01]{BBN01} {\scshape V.~Balaji, I.~Biswas and 
D.~S.~Nagaraj}: {\em Principal bundles over projective manifolds with 
parabolic structure over a divisor}, Tohoku Math.\ Jour.
{\bfseries 53}, 337--367 (2001).

\bibitem[BBN03]{BBN03} {\scshape V.~Balaji, I.~Biswas and 
D.~S.~Nagaraj}: {\em Ramified $G$-bundles as parabolic bundles}, 
Jour.\ Ramanujan Math.\ Soc.\ {\bfseries 18}, 123--138 (2003).

\bibitem[Bq97]{Bq97} {\scshape O.~Biquard}: {\em Fibr\'es de Higgs et 
connexions int\'egrables: Le cas logarithmique 
(diviseur lisse)}, Ann.\ Sci.\ \'Ec.\ Norm.\ {\bfseries 30}, 41--96 (1997).

\bibitem[Bi97a]{Bi97a} {\scshape I.~Biswas}: {\em Chern classes for 
parabolic bundles}, Jour.\ Math.\ Kyoto Univ.\ {\bfseries 37}, 
597--613 (1997).

\bibitem[Bi97b]{Bi97b} {\scshape I.~Biswas}: {\em Parabolic bundles as 
orbifold bundles}, Duke Math.\ Jour.\ {\bfseries 88}, 305--325 (1997).

\bibitem[Bi06]{Bi06} {\scshape I.~Biswas}: {\em Connections on a 
Parabolic Principal Bundle over a Curve}, Can.\ Jour.\ Math.\ 
{\bfseries 58}, 262--281 (2006).

\bibitem[Bi08]{Bi08} {\scshape I.~Biswas}: {\em Parabolic principal 
Higgs bundles}, Jour.\ Ramanujan Math.\ Soc.\ {\bfseries 23}, 
311--325 (2008).

\bibitem[BS09]{BS} {\scshape I.~Biswas and G. Schumacher}:
{\em Yang--Mills equation for stable Higgs sheaves}, Inter.
Jour. Math. {\bfseries 20}, 541--556 (2009).

\bibitem[Do87]{Do} {\scshape S. K. Donaldson}: {\em Infinite 
determinants, stable bundles and curvature}, Duke Math. Jour.
{\bfseries 54}, 231--247 (1987).

\bibitem[Hu95]{Hu} {\scshape J. E. Humphreys}: \textit{Conjugacy
classes in semisimple algebraic groups}, Mathematical Surveys and
Monographs, 43. American Mathematical Society, Providence, RI, 1995.

\bibitem[MY92]{MY92} {\scshape M.~Maruyama and K.~Yokogawa}: {\em 
Moduli of parabolic stable sheaves}, Math.\ Ann.\ {\bfseries 293}, 
77--99 (1992).

\bibitem[No76]{No76} {\scshape M.~V.~Nori}: {\em On the representations 
of the fundamental group}, Compos.\ Math.\ {\bfseries 33}, 29--41 (1976).

\bibitem[RS88]{RS88} {\scshape A.~Ramanathan and S.~Subramanian}: {\em Einstein-Hermitian
connections on principal bundles and stability}, Jour.\ Reine 
Angew.\ Math.\ {\bfseries 390}, 21--31 (1988).

\bibitem[Se77]{Se77} {\scshape C.~S.~Seshadri}: {\em Moduli of vector 
bundles on curves with parabolic structures}, Bull.\ Am.\ Math.\ Soc.\ 
{\bfseries 83}, 124--126 (1977).

\bibitem[Si88]{Si88} {\scshape C. T. Simpson}: {\em Constructing 
variations of Hodge structure using Yang-Mills theory and 
applications to uniformization},  Jour. Amer. Math. Soc. 
{\bfseries 1}, 867--918 (1988).

\bibitem[Si92]{Si} {\scshape C. T. Simpson}: {\em Higgs bundles and
local systems}, Inst. Hautes \'Etudes Sci. Publ. Math.
{\bfseries 75}, 5--95 (1992).

\bibitem[UY86]{UY} {\scshape K. Uhlenbeck and S.-T. Yau}: {\em On
the existence of Hermitian--Yang--Mills connections on stable
vector bundles}, Commun. Pure Appl. Math.
{\bfseries 39}, 257--293 (1986).

\end{thebibliography}
\end{document}